\documentclass{article}

\usepackage{amsthm}
\usepackage{amsmath,amscd}
\usepackage{amsfonts}
\usepackage{amssymb}
\usepackage[T1]{fontenc}

\newtheorem{theorem}{Theorem}[section]
\newtheorem{lemma}[theorem]{Lemma}
\newtheorem{e-proposition}[theorem]{Proposition}
\newtheorem{corollary}[theorem]{Corollary}
\newtheorem{e-definition}[theorem]{Definition\rm}
\newtheorem{conjecture}[theorem]{Conjecture}
\newtheorem{remark}{\it Remark\/}

\DeclareMathOperator{\codim}{codim}

\title{Some Remarks on Generic Complete Intersection Varieties}
\author{Damian Brotbek}

\begin{document}
\maketitle

\begin{abstract}
We prove an existence theorem for jet differentials on complete intersection varieties that generalizes a theorem of S. Diverio. We also show that one can readily deduce hyperbolicity for generic complete intersections of high codimension and high multidegree from the papers $\cite{DMR10}$ and $\cite{D-T10}$. And finally we prove the numerical aspect of a conjecture of O. Debarre.
\end{abstract}

\section{Introduction}
We will consider three closely related aspects of complete intersection varieties of high multidegree, the existence of jet differentials, hyperbolicity and the positivity of the cotangent bundle. The first aspect is motivated by two theorems of S. Diverio (see \cite{Div08} and \cite{Div09}).
In \cite{Div09}, he proved a nonvaninshing theorem for jet differentials on hypersurfaces of high degree.

\begin{theorem}\label{nonvanishing}
\emph{(\cite{Div09} Theorem 1)}
Let $X\subset \mathbb{P}^{n+1}$ a smooth projective hypersurface and $a\in \mathbb{N}$. Then there exists a positive integer $d$ such that 
$$H^0(X, E_{k,m}\Omega_X\otimes \mathcal{O}_X(-a))\neq 0$$
if $k\geqslant n$, $\deg(X)\geqslant d$ and $m$ large enough.
\end{theorem} 

He also proved a vanishing theorem for jet differentials on complete intersection varieties.

\begin{theorem}\label{vanishing} 
\emph{(\cite{Div08} Theorem 7)}
Let $X$ be a complete intersection variety in $\mathbb{P}^N$ of dimension $n$ and codimension $c$.
For all $m\geqslant 1$ and $1\leqslant k < \lceil\frac{n}{c}\rceil$ one has
$$H^0(X,E_{k,m}\Omega_X)=0.$$ 
\end{theorem}

It seems therefore natural to look for the nonvanishing of $H^0(X,E_{\lceil\frac{n}{c}\rceil,m}\Omega_X)$ when $X$ is a smooth complete intersection of dimension $n$ and codimension $c$ of high multidegree. This is the content of Theorem $\ref{thmjet}$.
The second aspect is motivated by the results from \cite{DMR10} and \cite{D-T10}. They prove that given a generic hypersurface $X$ of high degree in $\mathbb{P}^N$ there exists an algebraic subspace of codimension two in $X$ that contains all entire curves of $X$, (the existence of such a algebraic subspace is proven in \cite{DMR10}, and the "codimension two" raffinement is proven in \cite{D-T10}). With this, one can easily deduce better results for generic complete intersections of higher codimensions, simply by "moving" the hypersurfaces that we are intersecting. In particular we get hyperbolicity when $3\codim_{\mathbb{P}^N}(X) \geqslant  \dim(X)$, for a generic complete intersection variety $X$ of high multidegree. This is the content of Corollary $\ref{HCI}$. The last aspect is motivated by a conjecture of Olivier Debarre. In \cite{Deb05}, he conjectures that generic complete intersection varieties of high codimension in $\mathbb{P}^N$ and of high multidegree have ample cotangent bundle. In Theorem $\ref{numpos}$ we prove that the Chern classes satisfy all the positivity conditions that one might expect (according to \cite{F-L83}). Note that all those results would be immediat consequences of a more general conjecture raised in \cite{D-T10}.

\begin{conjecture} Let $X\subset \mathbb{P}^N$ be a generic complete intersection variety of dimension $n$ and codimension $c$, then $E_{k,m}\Omega_X$ is ample as soon as $k\geqslant \lceil\frac{n}{c}\rceil$.
\end{conjecture}

In the last section we consider another side of the positivity of the cotangent bundle. That is, given an embedding $X\subseteq \mathbb{P}^N$ we know that $\Omega_X\otimes \mathcal{O}_X(2)$ is nef (since it is globally generated) but not ample in general. Therefore one can ask under which conditions we get ampleness. This has a really neat geometric interpretation, more precisely ampleness is equivalent to the nonexistence of lines in $X$. This is the content of Proposition $\ref{droite}$.

 Concerning our notations, we will work over $\mathbb{C}$, the field of complex numbers. If $E\to X$ is a vector bundle over some variety $X$, $\pi_E: \mathbb{P}(E)\rightarrow X$ will denote the projective bundle of one dimensional quotients of $E$ and $\mathcal{O}_{\mathbb{P}(E)}(1)$ the tautological quotient line bundle on $\mathbb{P}(E)$.


\section{Preliminaries}

\subsection{Segre Classes and Higher order jet spaces}
We start by giving the definition of the Segre classes associated to a vector bundle.
If $E$ is a rank $r$ complex vector bundle on $X$ and $p:\mathbb{P}(E)\to X$ the projection, then the Segre classes of $E$ are defined by $$s_i(E):=p_*c_1(\mathcal{O}_{\mathbb{P}(E)}(1))^{r-1+i}.$$ 
(Note that it is denoted $s_i(E^*)$ in $\cite{Ful98}$). It is straightforward to check that for any line bundle $L\to X$
\begin{equation}\label{segredroite}
s_i(E\otimes L)= \sum_{j=0}^{i} \binom{r-1+i}{i-j}s_j(E)c_1(L)^{i-j}.
\end{equation} 
 
Recall that the total Segre class is the formal inverse of the total Chern class of the dual bundle, $s(E)=c(E^*)^{-1}$. Therefore total Segre classes satisfy Whitney's formula for vector bundle exact sequences.

Now we briefly recall the construction of higher order jet spaces , details can be found in \cite{Dem00} and \cite{Mou10}, we will follow the presentation of $\cite{Mou10}$. Let $X\subset \mathbb{P}^N$ be a projective variety of dimension $n$, for all $k\in \mathbb{N} $ we can construct a variety $X_k$ and a rank $n$ vector bundle $\mathcal{F}_k$ on $X_k$. Inductively,  $X_0:=X$ and $\mathcal{F}_0:=\Omega_X$. Let $k\geqslant 0$ suppose that $X_k$ and $\mathcal{F}_k$ are constructed, then  $X_{k+1}:=\mathbb{P}(\mathcal{F}_{k})\overset{\pi_{k,k+1}}{\longrightarrow} X_k$ and $\mathcal{F}_{k+1}$ is the quotient of $\Omega_{X_{k+1}}$ defined by the following diagram.\\

$\begin{CD}
  @.            0                @.                 0          @.                                 @.          \\
@.            @VVV                                @VVV                          @.                          @.\\
  @.         S_{k+1}             @=              S_{k+1}       @.                                 @.          \\   
@.            @VVV                                @VVV                          @.                          @.\\ 
0 @>>>\pi^*_{k,k+1}\Omega_{X_k} @>>>     \Omega_{X_{k+1}}       @>>>       \Omega_{X_{k+1}/X_{k}}   @>>>    0\\  
@.            @VVV                                @VVV                          @|                          @.\\ 
0 @>>>\mathcal{O}_{X_{k+1}}(1)  @>>>    \mathcal{F}_{k+1}      @>>>       \Omega_{X_{k+1}/X_{k}}   @>>>     0 \\
@.            @VVV                                @VVV                          @.                          @.\\
  @.           0                @.                 0            @.                                 @.         \\ \\
\end{CD}$

For all $k > j \geqslant 0$ we will denote $\pi_{j,k}=\pi_{j,j+1}\circ \cdots \circ \pi_{k-1,k}: X_k\to X_j$, $\pi_k:=\pi_{0,k}$ and $E_{k,m}\Omega_X={\pi_k}_*\mathcal{O}_{X_k}(m)$. The bundles $E_{k,m}\Omega_X$ have important applications to hyperbolicity problems. 

Note that $n_k:=\dim(X_k)=n+k(n-1)$. 
If we have a k-uple of integers $(a_1,\cdots ,a_k)$ we will write 
$$\mathcal{O}_{X_k}(a_1,\cdots, a_k)=\pi_{1,k}^{*}\mathcal{O}_{X_1}(a_1)\otimes \cdots \otimes \mathcal{O}_{X_k}(a_k)$$


We define $s_{k,i}:=s_i({\mathcal{F}_k})$, $s_i:=s_i(\Omega_X)$, $u_k:=c_1(\mathcal{O}_{X_k}(1))$, $h_{\mathbb{P}^N}:=c_1(\mathcal{O}_{\mathbb{P}^N}(1))$, $h:={h_{\mathbb{P}^N}}_{|X}$ and $\mathcal{C}_k(X)=\mathbb{Z}\cdot u_k \oplus \cdots \oplus \mathbb{Z}\cdot u_1 \oplus \mathbb{Z}\cdot h \subset NS^1(X_k)$.
To ease our computations we will also adopt the following abuses of notations, if $k>j$ we will write $u_j$ the class on $X_k$ instead of $\pi^*_{j,k}u_j$ and similarly $s_{j,i}$ instead of $\pi^*_{j,k}s_{j,i}$. This should not lead to any confusion.

Now the horizontals exact sequences in the diagram, Whitney's formula and formula $\ref{segredroite}$ one easily derives (as in $\cite{Mou10}$)
the recursion formula
\begin{equation}\label{formulasegre}
s_{k,\ell}=\sum_{j=0}^{\ell}M_{\ell,j}^ns_{k-1,j}u_k^{\ell-j},
\end{equation}
where $M_{\ell,j}^n=\sum_{i=0}^{\ell-j}(-1)^i\binom{n-2+i+j}{i}$, in particular $M_{\ell,\ell}^n=1$.

\begin{lemma}\label{labase}
Let $k\geqslant 0$, $a\geqslant 0$, $\ell\geqslant 0$, take $\ell$ positive integers $i_1,\cdots, i_\ell$ and $m$ divisor classes $\gamma_1,\cdots, \gamma_m \in \mathcal{C}_k(X)$ such that $i_1+\cdots +i_\ell+m+a=n_k$. Denote $\gamma_q:=\alpha_{q,0}h+\sum_{i} \alpha_{q,i} u_i$. Then

$$ \int_{X_k}s_{k,i_1}\cdots s_{k,i_\ell} \gamma_1\cdots \gamma_m h^a =
\sum_{j_1,\cdots,j_{k+\ell},b} Q_{j_1,\cdots,j_{k+\ell},b}\int_X s_{j_1}\cdots s_{j_{k+\ell}} h^{a+b}
$$
Where in each term of the sum we have, $b\geqslant 0$ and the $Q_{j_1,\cdots,j_{k+\ell},b}$'s are polynomials in the $\alpha_{q,i}$'s whose coefficients are independant of $X$. Moreover up to reordering of the $j_p$'s one has $j_1\leqslant i_1,\cdots j_\ell\leqslant i_\ell$. 
\end{lemma}

\begin{proof}
This is an immediate induction on $k$. The result is clear for $k=0$. Now suppose it is true for some $k>0$ and take $m$ divisors, $\gamma_1,\cdots, \gamma_m \in \mathcal{C}_{k+1}(X)$  on $X_{k+1}$. Let $\gamma_q:=\alpha_{q,0}h+\sum_{i} \alpha_{q,i} u_i$. Then using recursion formula \ref{formulasegre} and expanding, we get
\begin{eqnarray}
\int_{X_{k+1}}\!\!\!\!\!\!\!s_{k+1,i_1}\cdots s_{k+1,i_\ell}\gamma_1\cdots \gamma_m h^a\!\!\!\!\!&=&\!\!\!\!\! \sum P_{j_i,\cdots, j_\ell,b}^{p_1,\cdots, p_{k+1}}\!\!\!\int_{X_{k+1}}\!\!\!\!\!\!\!s_{k,j_1}\cdots s_{k,j_\ell}u_{k+1}^{p_{k+1}}\cdots u_1^{p_1}h^{a+b} \nonumber \\
&=&\!\!\!\!\! \sum P_{j_i,\cdots, j_\ell,b}^{p_1,\cdots, p_{k+1}}\!\!\!\int_{X_{k}}\!\!\!\!\!s_{k,j_1}\cdots s_{k,j_\ell} s_{k,r} u_{k}^{p_{k}}\cdots u_1^{p_1}h^{a+b} \nonumber
\end{eqnarray}
where in each term of the sum, $r=p_{k+1}-(n-1)$ and moreover, thanks to formula $\ref{formulasegre}$ one has $j_1\leqslant i_1,\cdots, j_\ell\leqslant i_\ell$. Note also that the $P_{j_i,\cdots, j_\ell,b}^{p_1,\cdots, p_{k+1}}$'s are polynomials in the $\alpha_{i,j}$'s but their coefficients do not depend on $X$. Now we can conclude by induction hypothesis.
\end{proof}


\subsection{Numerical positivity}
Following Fulton $\cite{Ful98}$ we recall definitions concerning Schur polynomials.
Let $c_1, c_2, c_3,...$ be a sequence of formal variables.
Let $\ell$ be a positive integer and $\lambda=(\lambda_1,...,\lambda_{\ell})$ be a partition of ${\ell}$. We define the Schur polynomial associated to $c=(c_i)_{{i\in \mathbb{N}}}$ and $\lambda$ to be:
\begin{eqnarray}
\Delta_{\lambda}(c):= \det \left[(c_{\lambda_i+j-i})_{1\leqslant i,j\leqslant {\ell}}\right]\nonumber
\end{eqnarray}
For exemple, $\Delta_{1}(c)=c_1$, $\Delta_{2,0}(c)=c_2$ and $\Delta_{(1,1)}(c)=c_1^2-c_2$.

Now consider two sequences of formal variables, $c_1, c_2, c_3,...$ and $s_1, s_2, s_3,...$ satisfying the relation:
\begin{eqnarray}
(1+c_1 t + c_2 t^2+ ...)\cdot(1-s_1 t +s_2 t^2 - ...)=1 \label{chernsegre}
\end{eqnarray}

Note that relation (\ref{chernsegre}) is satisfied when $c_i=c_i(E)$ are the Chern classes of a vector bundle $E$ over a variety $X$ and $s_i=s_i(E)$ are its Segre classes.

The proof of the following crucial combinatorial result can be found in $\cite{Ful98}$.

\begin{lemma}\label{lemmecombi}
With the same notations.
Let $\bar{\lambda}$ be the conjugate partition of $\lambda$, then $\Delta_{\lambda}(c)=\Delta_{\bar{\lambda}}(s)$.
\end{lemma}

Let $E$ be a vector bundle of rank $r$  over a projective variety $X$ of dimension~$n$. 

\begin{e-definition}
We will say that $E$ is numerically positive if for any subvariety $Y \subseteq X $ and for any partition $\lambda $ of $ {\ell}=dim(Y) $ one has $\int_{Y}\Delta_{\lambda}(c(E))>0$.
\end{e-definition}

This definition is motivated by a theorem of Fulton and Lazarsfeld $\cite{F-L83}$ which gives numerical consequences of ampleness.
\begin{theorem}
If $E$ is ample then $E$ is numerically positive. Moreover the Schur polynomials are exactly the relevant polynomials to test ampleness numerically.
\end{theorem}
Note that the converse is false, for example the bundle $\mathcal{O}_{\mathbb{P}^1}(2)\oplus \mathcal{O}_{\mathbb{P}^1}(-1)$ over $\mathbb{P}^1$ is numerically positive but not ample (the problem being the lack of subvarieties of dimension two to test $c_2$). See $\cite{Ful76}$ for a more interesting example.


\subsection{Segre classes for complete intersections}\label{segre}

From now on, we will take $X=H_1\cap\cdots \cap H_c\subset \mathbb{P}^N$ a complete intersection of dimension $n$ and $H_i=(\sigma_i=0)$ with $\sigma_i \in H^0(\mathbb{P}^N,\mathcal{O}_{\mathbb{P}^N}(d_i))$ and $d_i \in \mathbb{N}$. Note that $n+c=N$. Let $\kappa:= \lceil\frac{n}{c}\rceil$ and take $b$ such that $n=(\kappa -1)c+b$, observe that $0<b\leqslant c$. To simplify some of our formulas, we will also define $\hat{\imath}:=i+n-1$ so that ${\pi_{k-1,k}}_*u_k^{\hat{\imath}}=s_{k-1,i}$.
Moreover as we will be interested in the asymptotic behavior of polynomials in $\mathbb{Z}[d_1,\cdots , d_c]$ we need some more notations. Let $P\in\mathbb{Z}[d_1,\cdots , d_c]$, $\deg{P}$ denotes the degree of the polynomial and $P^{dom}$ the homogenous part of $P$ of degreee $\deg P$ . We will write $P=o(d^k)$ if $\deg{P}<k$ and if $Q\in \mathbb{Z}[d_1,\cdots , d_c]$ is another polynomial we will write $P\sim Q$ if $P^{dom}=Q^{dom}$ and $P\gtrsim Q$ if $P^{dom}\geqslant Q^{dom}$.  Moreover some of our computations will take place in the chow ring $A^*(X)$ so we introduce some more notations. If $P$ is a polynomial in $\mathbb{Z}[d_1,\cdots , d_c, h]$ homogenous of degree $k$ in $h$ we will write $\tilde{P}$ for the unique polynomial in $\mathbb{Z}[d_1,\cdots , d_c]$ satisfying $P=\tilde{P}h^k$.
 
First we compute the Segre classes of the twisted cotangent bundle of a complete intersection in $\mathbb{P}^N$. Let $m\in \mathbb{Z}$. The twisted Euler exact sequence
\begin{eqnarray}
0 \to \mathcal{O}_{\mathbb{P}^N}(-m) \to \mathcal{O}_{\mathbb{P}^N}^{\oplus N+1}(1-m) \to T\mathbb{P}^N(-m) \to 0 \nonumber
\end{eqnarray}
yields:
\begin{eqnarray}
c(T\mathbb{P}^N(-m))=\frac{c(\mathcal{O}_{\mathbb{P}^N}^{\oplus N+1}(1-m))}{c(\mathcal{O}_{\mathbb{P}^N}(-m))}
=\frac{(1+(1-m)h_{\mathbb{P}^N})^{N+1}}{(1-mh_{\mathbb{P}^N})} \nonumber
\end{eqnarray}
The (twisted) normal bundle exact sequence 

\begin{eqnarray}
0 \to TX(-m) \to T\mathbb{P}^N_{|X}(-m) \to \bigoplus_{i=1}^c \mathcal{O}_{X}(d_i-m)\to 0 \nonumber
\end{eqnarray}
yields 
\begin{eqnarray}
c(TX(-m))=\frac{c(T\mathbb{P}^N_{|X}(-m))}{c(\bigoplus_{i=1}^c \mathcal{O}_{X}(d_i-m))}
=\frac{(1+(1-m)h)^{N+1}}{(1-mh)\prod_{i=1}^c(1+(d_i-m)h)} \nonumber
\end{eqnarray}
Therefore:
\begin{eqnarray}
s(\Omega_X(m))=(1\!-\!(1\!-\!m)h\!+\!(1\!-\!m)^2h^2\!-...)^{n+1}(1\!-\!mh)\prod_{i=1}^c(1\!+\!(d_i\!-\!m)h) \nonumber
\end{eqnarray}

Expanding the right hand side as a polynomial in $\mathbb{Z}[d_1,\cdots,d_c,h]$ we see that for $\ell \geq c$ we have $\deg(\tilde{s_\ell})=c$ and that for ${\ell}\leqslant c$ we have
\begin{eqnarray}\label{dominant}
s_{\ell}^{dom}(\Omega_X(m)) =\sum_{j_1<...<j_{\ell}} d_{j_1}...d_{j_{\ell}}h^{\ell}= c_{\ell}^{dom}(\bigoplus_{i=1}^c \mathcal{O}_{X}(d_i-m))  = c_{\ell}(\bigoplus_{i=1}^c \mathcal{O}_{X}(d_i)).
\end{eqnarray}

In particular, for $\ell \leqslant c$, we have $\deg(\tilde{s_\ell})=\ell$ and $\tilde{s_\ell}^{dom}>0$.
\begin{remark}
If $n\leqslant c$ we have that equality \ref{dominant} holds for all ${\ell}\in \mathbb{Z}$ (the case ${\ell}>n$ is obvious since both side of the equality vanish by a dimension argument).
\end{remark}

With this we can give some estimates for some intersection products on $X$.
\begin{lemma}\label{intersection}
\begin{enumerate}
\item Take $0\leqslant i_1\leqslant \cdots \leqslant i_k $, $\ell>0$ such that $i_1+\cdots +i_k+\ell=n$ then
$$\deg\left( \int_X s_{i_1}\cdots s_{i_k}h^{\ell} \right)<N$$ \label{intersection1}
\item Take $0\leqslant i_1\leqslant \cdots \leqslant i_k $ such that $i_1+\cdots +i_k=n$ then
 $\int_{X}s_{i_1}\cdots s_{i_k}$ is of degre $N$ if and only if $i_k\leqslant c$.

\item Take $0\leqslant i_1 \leqslant \cdots \leqslant i_{\kappa}$.
If $i_1< b$ or if $i_1 =b$ and $i_j< c$ for some $j>1$ then  $$\deg \left(\int_{X}s_{i_1}\cdots s_{i_{\kappa}}\right)<N.$$

\end{enumerate}
\end{lemma}

\begin{proof}
First note that for $0\leqslant {\ell} \leqslant n$, since $\deg \int_X h^n=c $,  
$$\deg\left( \int_X s_{i_1}\cdots s_{i_k}h^{\ell} \right)=\deg{\tilde{s}_{i_1}}+\cdots +\deg{\tilde{s}_{i_k}} +c.$$
Recall also that $\deg(\tilde{s_{i}})=i$ if $i\leqslant c$ and $\deg(\tilde{s_{i}})=c$ if $i > c$.
Now the first point is clear, let ${\ell} >0$
\begin{eqnarray}
\deg\left( \int_X s_{i_1}\cdots s_{i_k}h^{\ell} \right)&\leqslant & i_1 +\cdots + i_k +c <n+c=N \nonumber
\end{eqnarray}
To see the second point, if $i_{k}\leqslant c$ so are all the $i_j$ and therefore $$\deg \left(\int_{X} s_{i_1}\cdots s_{i_{k}}\right)=i_1 + \cdots +i_{k} +c = N.$$ Conversely if $i_{k}>c$ then $$\deg \left(\int_{X}s_{i_1}\cdots s_{i_{k}}\right)=i_1 + \cdots +i_{k-1}+c +c <N.$$
The last point is an easy consequence of the second one thanks to the equality $n=(\kappa -1)c +b$.
\end{proof}
\section{Jet differentials on complete intersection varieties}

In \cite{Div08} Diverio proves Theorem \ref{vanishing} which is a result concerning the nonexistence of global jet differentials on complete intersections. Here we will prove that this theorem is optimal. Recall that $\kappa$ denotes the ratio~$\lceil \frac{n}{c} \rceil$.

\begin{theorem}\label{thmjet}
Fix $a\geqslant 0$. There exists a constant $\delta$ such that, if for all $i$, $d_i\geqslant \delta$ then $\mathcal{O}_{X_\kappa}(1)\otimes \pi_\kappa^* \mathcal{O}_X(-a)$ is big on $X_{\kappa}$. In particular
$$H^0(X,E_{\kappa,m}\Omega_X(-a))\neq 0$$
when $m\gg 0$.
\end{theorem} 

\begin{remark}
It will be clear during the proof that all computations could be made explicit, at least in small dimensions, so that the $\delta$ in the theorem is, in theory at least, effective. But we will not consider this problem here. 
\end{remark}

\begin{remark}
It is also natural to ask if under the hypothesis of the theorem, $\mathcal{O}_{X_k}(1)\otimes \pi_k^* \mathcal{O}_X(-a)$ is big if $k\geqslant \kappa$. In fact this is true and the proof is exactly the same.  
\end{remark}

As in $\cite{Div09}$ we introduce a nef line bundle $L_k\in \mathcal{C}_k(X)$ for all $k\geqslant 1$.

$$L_k:=\mathcal{O}_{X_k}(2\cdot 3^{k-2},\cdots , 2\cdot 3,2,1)\otimes\pi^*
\mathcal{O}_X(2\cdot 3^{k-1})$$
We can write its first chern class
$${\ell}_k:=c_1(L_k)=u_k +\beta_k$$
where $\beta_k$ is a class that comes from $X_{k-1}$.
Now we can state the main technical lemma. This is just the combination of
lemmas $\ref{labase}$ and $\ref{intersection}$.

\begin{lemma}\label{technical}
With the above notations we have the following estimates.
\begin{enumerate}
\item Take $k\geqslant 1$ and $\gamma_1,\cdots ,\gamma_{n_k-1} \in \mathcal{C}_{k}(X)$. Then 
$$\int_{X_k} \gamma_1 \cdots \gamma_m h=o(d^N).$$
\item Take $\gamma_1,\cdots ,\gamma_p \in \mathcal{C}_{k}(X)$ and $0\leqslant i_1 \leqslant \cdots \leqslant i_{q}$ such that $p+\sum i_j=n_k$. If $i_1< b$ or if $i_1 =b$ and $i_j< c$ for some $j>1$ then we have the following estimates, 
\begin{eqnarray}
\int_{X_k}s_{k,i_1}\cdots s_{k,i_q}\gamma_1\cdots \gamma_p&=&o(d^N)\label{estim1}\\
\int_{X_{k}}s_{k-1,i_1}\cdots s_{k-1,i_q}\gamma_1\cdots \gamma_p&=&o(d^N)\label{estim}
\end{eqnarray}
\item Take $0<k<\kappa$. Then
$$\int_{X_k} s_{k,b}s_{k,c}^{\kappa-k-1}{\ell}_{k}^{\hat{c}}\cdots {\ell}_1^{\hat{c}}= \int_{X_{k-1}} s_{k-1,b}s_{k-1,c}^{\kappa-k}{\ell}_{k-1}^{\hat{c}}\cdots {\ell}_1^{\hat{c}}+o(d^N).$$

\end{enumerate}
\end{lemma}

\begin{proof}

1) is an immediate consequence of Lemma $\ref{intersection}.1$ and Lemma $\ref{labase}$.
Similarly for 2), thanks to Lemma $\ref{labase}$ we write
$$\int_{X_k}s_{k,i_1}\cdots s_{k,i_{q}}\gamma_1\cdots \gamma_p= \sum_{j_1,\cdots,j_{k+q}} Q_{j_1,\cdots,j_{k+q}}\int_X s_{j_1}\cdots s_{j_{k+q}} h^{a}$$

where $a\geqslant 0$ and moreover we know that in each term of this sum either $j_1<b$ or $j_p< c$ for some $p>0$, thus we can apply Lemma $\ref{intersection}.1$ (if $a>0$) or $\ref{intersection}.2$ (if $a=0$). From this one can easily deduce formula \ref{estim}, write $\gamma_i=a_iu_i+\beta_i$ where $\beta_i\in \mathcal{C}_{k-1}(X)$. 
\begin{eqnarray}
\int_{X_{k}}s_{k-1,i_1}\cdots \!\!\!\! &s_{k-1,i_q}&\!\!\!\!\gamma_1\cdots \gamma_p = \int_{X_{k}}s_{k-1,i_1}\cdots s_{k-1,i_q}(a_1u_1+\beta_1)\cdots (a_1u_1+\beta_p)\nonumber \\
&=&\!\!\!\!\!\!\!\!\!\! \sum_{I\subseteq{\{1,\cdots,p \}}}{\left( \prod_{i\in I}{a_i} \right)}\int_{X_{k}}s_{k-1,i_1}\cdots s_{k-1,i_q}u_k^{|I|}\prod_{i\notin I}{\beta_i}\nonumber \\
&=&\!\!\!\!\!\!\!\!\!\! \sum_{I\subseteq{\{1,\cdots,p \}}}{\left( \prod_{i\in I}{a_i} \right)}\int_{X_{k-1}}\!\!\!\!\!s_{k-1,i_1}\cdots s_{k-1,i_q}\cdot s_{k-1,|I|-n+1}\prod_{i\notin I}{\beta_i}\nonumber
\end{eqnarray}
Then we conclude by applying formula $\ref{estim1}$.

To see 3), 
\begin{eqnarray}
\int_{X_k} \!\!\!s_{k,b}s_{k,c}^{\kappa - k -1} {\ell}_{k}^{\hat{c}}\cdots {\ell}_{1}^{\hat{c}}\!\!\!\!&=&\!\!\!\! \int_{X_k}\!\! \left(\sum_{i=0}^{b}M_{b,i}^{n} s_{k-1,i}u_{k}^{b-i} \right)\!\! \left( \sum_{i=0}^{c}M_{c,i}^{n} s_{k-1,i}u_{k}^{c-i} \right)^{\kappa-k-1}\!\!\!\!\!\!\!\!\!\!\!\!\!\!\!\!\! {\ell}_{k}^{\hat{c}}\cdots {\ell}_{1}^{\hat{c}}\nonumber \\ 
                 &=& \!\!\!\! \int_{X_k}  s_{k-1,b}s_{k-1,c}^{\kappa - k -1} {\ell}_{k}^{\hat{c}}\cdots {\ell}_{1}^{\hat{c}}+o(d^N)\nonumber
\end{eqnarray}
By expanding and using formula \ref{estim} in each term of the obtained sum. Now,
\begin{eqnarray}
\int_{X_k}  s_{k-1,b}s_{k-1,c}^{\kappa - k -1} {\ell}_{k}^{\hat{c}}\cdots {\ell}_{1}^{\hat{c}}+o(d^N) \!\!\!\!&=&\!\!\!\! \int_{X_k} \!\!\!\! s_{k-1,b}s_{k-1,c}^{\kappa - k -1} (u_{k}+\beta_k)^{\hat{c}}{\ell}_{k-1}^{\hat{c}}\cdots {\ell}_{1}^{\hat{c}}\nonumber \\
                 &=&\!\!\!\! \int_{X_k} \!\!\!\! s_{k-1,b}s_{k-1,c}^{\kappa - k -1}  \sum_{i=0}^{\hat{c}}\binom{\hat{c}}{i}u_k^{\hat{c}-i}\beta_k^i {\ell}_{k-1}^{\hat{c}}\cdots {\ell}_{1}^{\hat{c}}\nonumber \\
                 &=&\!\!\!\! \int_{X_{k-1}} \!\!\!\!\!\!\! s_{k-1,b}s_{k-1,c}^{\kappa - k -1}  \sum_{i=0}^{c}\binom{\hat{c}}{i}s_{k-1,i}\beta_k^i {\ell}_{k-1}^{\hat{c}}\cdots {\ell}_{1}^{\hat{c}}\nonumber \\
                 &=&\!\!\!\! \int_{X_{k-1}}  \!\!\!\!\!\!\!s_{k-1,b}s_{k-1,c}^{\kappa - k -1} s_{k-1,c}{\ell}_{k-1}^{\hat{c}}\cdots {\ell}_{1}^{\hat{c}} +o(d^N)\nonumber \\
                 &=&\!\!\!\! \int_{X_{k-1}}  \!\!\!\!\!\!\!s_{k-1,b}s_{k-1,c}^{\kappa - (k-1) -1} {\ell}_{k-1}^{\hat{c}}\cdots {\ell}_{1}^{\hat{c}} +o(d^N)\nonumber
\end{eqnarray}

\end{proof}
Recall also the following consequence of holomorphic Morse inequalities (see $\cite{Laz04}$).

\begin{theorem}
Let $Y$ be a smooth projective variety of dimension $n$ and let $F$ and $G$ be two nef divisors on $Y$. If $F^n>nG\cdot F^{n-1}$ then $F-G$ is big.
\end{theorem}

We are now ready to prove Theorem $\ref{thmjet}$. 
\begin{proof}
First we recall an argument from $\cite{Div09}$ to show that we just have to check that $\mathcal{O}_{X_k}(a_1,\cdots, a_k)\otimes \pi_k^*\mathcal{O}_X(-a)$ is big for some suitable $a_i$'s. We know (see $\cite{Dem00}$) that $D_k:=\mathbb{P}(\Omega_{X_{k-1}/X_{k-2}})\subset X_k$ is an effective divisor that satisfies the relation $\pi_{k-1,k}^*\mathcal{O}_{X_{k-1}}(1)=\mathcal{O}_{X_k}(1)\otimes \mathcal{O}_{X_k}(-D_k)$. From this an immediate induction shows that for any $k>1$ and any $k$-uple $(a_1,\cdots,a_k)$ we have 
$$\mathcal{O}_{X_k}(b_{k+1})=\mathcal{O}_{X_k}(a_1,\cdots, a_k)\otimes\pi_{2,k}^*\mathcal{O}_{X_2}(b_1D_2)\otimes\cdots \otimes\mathcal{O}_{X_k}(b_{k-1}D_k).$$
Where for all $j>0$, $b_j:=a_1+\cdots +a_{j}$. Thus when $0\leqslant b_j$ for all $0\leqslant j\leqslant k$ then $\pi_{2,k}^*\mathcal{O}_{X_2}(b_1D_2)\otimes\cdots  \otimes\mathcal{O}_{X_k}(b_{k-1}D_k)$ is effective, this means that, under this condition, to prove that $\mathcal{O}_{X_k}(1)\otimes \pi_k^*\mathcal{O}_X(-a)$ is big it is sufficient to show that $\mathcal{O}_{X_k}(a_1,\cdots, a_k)\otimes \pi_k^*\mathcal{O}_X(-a)$ is big.

 Let $D=F-G$ where $F:=L_{\kappa} \otimes \cdots \otimes L_1$, and $G=\pi_{\kappa}^*(\mathcal{O}_X(m+a))$ where $m\geqslant 0$ is chosen so that $F\otimes\pi_{\kappa}\mathcal{O}_X(-m)$ has no component coming from $X$. It is therefore sufficent to show that $D$ is big. To do so, we will apply holomorphic Morse inequalities to $F$ and $G$ (both nef).
We need to prove that 
$$F^{n_{\kappa}}>n_{\kappa}F^{n_{\kappa}-1}\cdot G.$$
Clearly, the right hand side has degree strictly less than $N$ in the $d_i$'s thanks to Lemma $\ref{technical}$
and therefore we just have to show that the left hand side is larger than a positive polynomial of degree $N$ in the $d_i$'s. Let $\alpha := c_1(\pi_{\kappa}^*(\mathcal{O}_X(a))$

\begin{eqnarray}
F^{n_{\kappa}}&=& \int_{X_\kappa}({\ell}_{\kappa}+\cdots +{\ell}_1-\alpha)^{n_{\kappa}} \nonumber \\
				 &=& \int_{X_\kappa}\sum_{i=0}^{n_{\kappa}}(-1)^i\binom{n_{\kappa}}{i} ({\ell}_{\kappa}+\cdots +{\ell}_1)^{n_{\kappa}-i}\alpha^i \nonumber \\
                 &=& \int_{X_\kappa}({\ell}_{\kappa}+\cdots +{\ell}_1)^{n_{\kappa}}+o(d^N)\nonumber
\end{eqnarray}
by applying Lemma \ref{technical}. But since all the $\ell_i$'s are nef,
\begin{eqnarray}
\int_{X_\kappa}({\ell}_{\kappa}+\cdots +{\ell}_1)^{n_{\kappa}}&\geqslant & \int_{X_\kappa} {\ell}_{\kappa}^{\hat{b}} \cdot {\ell}_{\kappa-1}^{\hat{c}}\cdots {\ell}_1^{\hat{c}}\nonumber \\
                 &=& \int_{X_{\kappa -1}}\!\!\!s_{\kappa -1,b} \cdot {\ell}_{\kappa-1}^{\hat{c}}\cdots {\ell}_1^{\hat{c}}+o(d^N)\nonumber
\end{eqnarray}
The last inequality is obtained by using Lemma \ref{technical}.2. Now an immediate induction proves that for all $k<\kappa$ one has
$$F^{n_{\kappa}} \geqslant \int_{X_k} s_{k,b}s_{k,c}^{\kappa-k-1} l_{k}^{\hat{c}}\cdots l_{1}^{\hat{c}}+o(d^N).$$

We just proved the case $k = \kappa -1$ and the other part of the induction is exactly the content of Lemma $\ref{technical}.3$.
Therefore,
$$F^{n_{\kappa}} \geqslant \int_{X}  s_bs_c^{\kappa-1}+o(d^N)$$
and we conclude applying Lemma $\ref{intersection}.2$
\end{proof}

\section{Algebraic Degeneracy and Hyperbolicity for Complete intersections}

Here we show that from the work of \cite{DMR10} and \cite{D-T10} one can straightforwardly deduce the hyperbolicity of generic complete intersections of high codimension and of high multi degree. We start by a definition.

\begin{e-definition}
Let $X$ be a projective variety, we define the algebraic degeneracy locus to be the Zariski closure of the union of all nonconstant entire curves $f:\mathbb{C}\to X$ 
$$dl(X):= \overline{\bigcup{f(\mathbb{C})}}.$$
\end{e-definition}

Now recall the main result they prove.

\begin{theorem}\label{D-T}
There exists $\delta \in \mathbb{N}$ such that if $H$ is a generic hypersurface of degree $d\geq \delta$, then there exists a proper algebraic subset $Y \subset H$ of codimension at least two in $H$ such that  $dl(H)\subset Y$.
\end{theorem}

\begin{remark}\label{zariski}
In fact they prove something slightly stronger. Consider the universal hypersurface of degree $d$ in $\mathbb{P}^N$
$$\mathcal{H}_d=\left\{(x,t)\in \mathbb{P}^N \times T^d\;\;\; /\;\;\; x\in H_{d,t} \right\}$$
where $T^d:=\mathbb{P}(H^0(\mathbb{P}^N,\mathcal{O}_{\mathbb{P}^N}(d))^*)$ and $H_{d,t}=(t=0)$. Denote $\pi_d$ the projection on the second factor. 
During the proof of Theorem $\ref{D-T}$ it is shown that for $d\gg 0$ there exists a open subset $U_d\subset T^d$ and an algebraic subset $\mathcal{Y}_d\subset \mathcal{H}_{d|U_d}\subset U_d\times\mathbb{P}^N$ such that for all $t\in U_d$, the fibre $Y_{d,t}$ has codimension 2 in $H_{d,t}$ and $dl(H_{d,t})\subset Y_{d,t}$.
\end{remark}

\begin{remark}
It should also be noted that the $\delta$ in the theorem is effective. A first bound, $\delta\leq 2^{(N-1)^5}$, is computed in $\cite{DMR10}$ and the optimale bound, $\delta=N+2$, is obtained by J. Merker in $\cite{Mer10}$.   
\end{remark}

We are going to use the standard action of $G:= Gl_{N+1}(\mathbb{C})$ on $\mathbb{P}^N$. For any $g\in G$ and any variety $X\subseteq \mathbb{P}^N$ we write $g\cdot X:=g^{-1}(X)$.
\begin{remark}\label{rkaction} 
Let $g\in G$ and $X \subset \mathbb{P}^N$ a projective variety. If $f:\mathbb{C}\to g\cdot X$ is a non constant entire curve then $g\circ f :\mathbb{C}\to X$ is a non constant entire curve, therefore $g\circ f (\mathbb{C})\subseteq Y$ and thus $f(\mathbb{C})\subseteq g\cdot Y$. This proves that $g\cdot dl(X)= dl(g\cdot X)$
\end{remark}

\begin{remark}\label{cap}
Note also that if $X_1$ and $X_2$ are two projective varieties in $\mathbb{P}^N$, then $dl(X_1\cap X_2)\subseteq dl(X_1)\cap dl(X_2)$. 
\end{remark}
We will combine these remarks with the following moving lemma.

\begin{lemma}\label{moving}
Let $V\subset \mathbb{P}^N$ and $W\subset \mathbb{P}^N$ be algebraic subsets such that $\dim (V)= n$ and $\dim (W)= m$. Then for a generic $g\in G$, we have  $$\dim (g\cdot V \cap W)= \max\left\{n+m-N,0\right\}.$$
\end{lemma}

\begin{proof}
We will proced by induction on $m$. The case $m=0$ is clear.
Now for the general case, let $H$ be a hyperplane of $\mathbb{P}^N$, such that $\dim(H\cap W)\leqslant m-1$ (such a hyperplan clearly exists). Now we take a generic $g\in G$, such that, by induction hypothesis, $\dim(g\cdot V \cap (W \cap H))= \max\left\{n+m-1-N,0\right\}$. Now argue by contradiction, obviously $\dim(g\cdot V \cap W)\geq \max\left\{n+m-N,0\right\}$. If $\dim(g\cdot V \cap W)> \max\left\{n+m-N,0\right\}$ then 
there exist an irreducible $Z\subset g\cdot V \cap W$ such that $\dim (Z)=\alpha>\max\left\{n+m-N,0\right\}$. But then $\dim (Z\cap H)\geqslant \alpha -1$, and since $Z\cap H \subset g\cdot V \cap (W\cap H)$, we get $\max\left\{n+m-N-1,0\right\} < \alpha-1 \leqslant \dim(Z\cap H) \leqslant \dim (g\cdot V \cap (W\cap H)) \leqslant \max\left\{n+m-1-N,0\right\}$, which yields the desired contradiction.
\end{proof}

\begin{corollary}\label{pseudogeneric}
Let $c\geqslant 1$, take $c$ hypersurfaces $H_1,\cdots ,H_c \subset \mathbb{P}^N$ satisfying the conclusions of Theorem $\ref{D-T}$ and take generic $g_2\cdots ,g_c\in G$. Consider $X:= H_1 \cap g_2\cdot H_2 \cap \cdots \cap g_c\cdot H_c$ (which is a smooth complete intersection). Then there is an algebraic subset $Y \subset X$ of dimension $N-3c$ in $X$, such that $dl(X)\subseteq Y $. In particular $X$ is hyperbolic if $c\geqslant \frac{N}{3}$. 
\end{corollary}

\begin{proof}
This is a simple induction on $c$. The case $c=1$ is just the content of Theorem $\ref{D-T}$. Let $c>0$ and suppose that $X_{c-1}:= H_1 \cap g_2\cdot H_2 \cap \cdots \cap g_c\cdot H_{c-1}$ is a smooth complete intersection that contains a proper subset $Y_{c-1}$ of dimension $N-3c+3$ in $X$, such that $Y_{c-1}$ contains the image of all the entire curves in $X$. By hypothesis there is an algebraic subset $Y_c\subset X_c$ of dimension $N-3$, such that $dl(H_c)\subseteq Y_c$.  Now for generic $g_c\in G$, $X_{c-1}\cap g_c\cdot H_c$ is a smooth complete intersection and moreover on can apply Lemma $\ref{moving}$ to $Y_{c-1}$ and $Y_c$ and thus
 
$$\dim (g_c\cdot Y_c  \cap Y_{c-1})= N-3+N-2c+1-N= N-3c.$$
Now, Remarks \ref{rkaction} and \ref{cap} yield $dl(X)\subseteq Y$.
\end{proof}

\begin{corollary}\label{HCI}
Let $X=H_1\cap \cdots \cap H_c \subset \mathbb{P}^N$, $H_i \in H^0(\mathbb{P}^N,\mathcal{O}_{\mathbb{P}^N}(d_i))$ be a generic smooth complete intersection such that $d_i\geq \delta $. Then there is an algebraic subset $Y \subset X$ of dimension $N-3c$ in $X$, such that $dl(X)\subseteq Y $. In particular $X$ is hyperbolic if $c\geqslant \frac{N}{3}$. 
\end{corollary}

\begin{proof}
With the notations of remark $\ref{zariski}$, consider $U=U_{d_1}\times \cdots \times U_{d_c}$, $\mathcal{Y}= \mathcal{Y}_{d_1}\cap \cdots \cap \mathcal{Y}_{d_1} $. It is well known that $\dim(Y_t)$ is an upper semi continuous function on $U$. Therefore we only need to show that there is no Zariski open subset in $U$ on which $\dim(Y_t)>N-3c$, but this is clear by corollary $\ref{pseudogeneric}$.
\end{proof}

\section{Numerical positivity of the cotangent bundle}
In $\cite{Deb05}$, O. Debarre proves that on an $N$ dimensional abelian variety, the intersection of at least $N/2$ generic hypersurfaces of high degree has ample cotangent bundle. He conjectures that the analoguous result holds in $\mathbb{P}^N$. We already proved some evidences toward this conjecture, with the notations of section \ref{segre}, Theorem \ref{thmjet} says in particular that if $c\geq n$ then $\mathcal{O}_{\mathbb{P}(\Omega_X)}(1)\otimes\pi_1^*\mathcal{O}_X(-a)$ is big (note that using Riemann-Roch Theorem, O. Debarre had already proved the bigness of $\mathcal{O}_{\mathbb{P}(\Omega_X)}(1)\otimes\pi_1^*\mathcal{O}_X(1)$). Moreover, as a particular case of Corollary \ref{HCI}, we know that if $X$ is also generic then $X$ is hyperbolic. As a last evidence toward this conjecture, we will prove the following.

\begin{theorem}\label{numpos}
With the same notations, take $a\in\mathbb{Z}$. If $c\geqslant n$, then there exists $D\in \mathbb{N}$ such that if $d_i > D$ for all $i$ then $\Omega_X(-a)$ is numericaly positive.
\end{theorem}

\begin{proof}
By Lemma \ref{lemmecombi} we have to check that for any subvariety $Y\subseteq X$ of dimension $\ell\leqslant n\leqslant c$ and for any partition $\lambda$ of $\ell$ one has $\int_Y \Delta_{\bar{\lambda}}(s(\Omega_X(-a)))>0$. Moreover, $\int_Y \Delta_{\bar{\lambda}}(s(\Omega_X(-a)))= \tilde{\Delta}_{\bar{\lambda}}(s(\Omega_X(-a)))\int_Yh^{\ell}$, thus we just have to check that $\tilde{\Delta}_{\bar{\lambda}}(s(\Omega_X(-a)))>0$ when the $d_i$'s are large enough, which is equivalent to $\tilde{\Delta}^{dom}_{\bar{\lambda}}(s(\Omega_X(-a)))>0$. Now the equality 
\begin{eqnarray}
 \tilde{\Delta}_{\bar{\lambda}}^{dom}(s(\Omega_X(-a)))=\det (\tilde{s}_{\lambda_i+j-i}^{dom}(\Omega_X(-a)))_{1\leqslant i,j\leqslant \ell}\label{equa}
\end{eqnarray}
holds if one can prove that the right hand side is non zero.
But by (\ref{dominant}) we find :
\begin{eqnarray}
\det (\tilde{s}_{\lambda_i+j-i}^{dom}(\Omega_X(-a)))_{1\leqslant i,j\leqslant l}&=&\det \left(\tilde{c}_{\lambda_i+j-i}\left(\bigoplus_{j=1}^k\mathcal{O}(d_j)\right)\right)_{1\leqslant i,j\leqslant l}\nonumber \\
&=&\tilde{\Delta}_{\bar{\lambda}}\left(c\left(\bigoplus_{j=1}^k\mathcal{O}(d_j)\right)\right)\nonumber
\end{eqnarray}
By applying the theorem of Fulton and Lazarsfeld to $\bigoplus_{j=1}^k\mathcal{O}_X(d_j)$ (which is ample if $d_i>0$) we find that this is striclty positive. 
This yields equality in \ref{equa}, and we get the desired result.
\end{proof}

\section{Ampleness of $\Omega_X(2)$}
In this last section we give a geometric interpretation of the ampleness of $\Omega_X(2)$. Fix a $N+1$-dimensional complexe vector space $V$, denote $\mathbb{P}^N=\mathbb{P}(V^*)$ the projectivized space of lines in $V$, $p : V \to \mathbb{P}(V^*)$ the projection, and $Gr(2,V)=Gr(2,\mathbb{P}^N)$ the space of vector planes in $V$ which is also the space of lines in $\mathbb{P}^N$. We will also consider the projection $\pi : \mathbb{P}(\Omega_{\mathbb{P}^N})\to \mathbb{P}^N$. 
The key point is the following lemma.
 
\begin{lemma}\label{Han}
There is a map $\varphi: \mathbb{P}(\Omega_{\mathbb{P}^N})\to \mathbb{P}(\Lambda^2 V^*)$ such that $\varphi^* \mathcal{O}_{\mathbb{P}(\Lambda^2 V^*)}(1)=\mathcal{O}_{\Omega_{\mathbb{P}^N}}(1)\otimes \pi^*\mathcal{O}_{\mathbb{P}^N}(2)$. Moreover this application factors through the Pl\"{u}cker embedding $Gr(2,\mathbb{P}^N)=Gr(2,V)\hookrightarrow \mathbb{P}(\Lambda^2 V^*)$. More precisely, an element $(x,[\xi])\in \mathbb{P}(\Omega_{\mathbb{P}^N})$ with $x\in\mathbb{P}^N$ and $\xi \in T_x \mathbb{P}^N$, gets maped to the unique line $\Delta$ in $\mathbb{P}^N$ satisfying $\xi\in T_x\Delta \subseteq T_x\mathbb{P}^N$. 
\end{lemma}

\begin{proof} 
Take the Euler exact sequence 
$$0 \to \mathcal{O}_{\mathbb{P}^N} \to V\otimes \mathcal{O}_{\mathbb {P}^N}(1)\to T\mathbb{P}^N \to 0$$ 
and apply $\Lambda^{N-1}$ to it to get the quotient
$$\Lambda^{N-1}V\otimes\mathcal{O}_{\mathbb{P}^N}(N-1)\to \Lambda^{N-1}T \mathbb{P}^N \to 0.$$
Now using the well known dualities, $\Lambda^{N-1}V=\Lambda^2V^*$ and $\Lambda^{N-1}T \mathbb{P}^N=\Omega_{\mathbb{P}^N}\otimes K_{\mathbb{P}^N}^*=\Omega_{\mathbb{P}^N}\otimes\mathcal{O}_{\mathbb{P}^N}(N+1)$, and tensoring everything by $\mathcal{O}_{\mathbb{P}^N}(1-N)$ we get 
$$\Lambda^2V^*\to \Omega_{\mathbb{P}^N}\otimes \mathcal{O}_{\mathbb{P}^N}(2)\to 0.$$
This yields the map $\varphi : {\mathbb{P}}(\Omega_{\mathbb{P}^N})=\mathbb{P}(\Omega_{\mathbb{P}^N}(2))\hookrightarrow {\mathbb{P}^N}\times \mathbb{P}(\Lambda^2V^*)\to \mathbb{P}(\Lambda^2V^*)$ such that $\varphi^*\mathcal{O}_{\mathbb{P}(\Lambda^2 V^2)}(1)=\mathcal{O}_{\Omega_{\mathbb{P}^N}(2)}(1)=\mathcal{O}_{\Omega_{\mathbb{P}^N}}(1)\otimes \pi^*\mathcal{O}_{\mathbb{P}^N}(2)$.

To see the geometric interpretation of this map, it suffices to backtrack through the previous maps. Take a point $x\in \mathbb{P}^N$ and a vector $0 \neq \xi \in T_x \mathbb{P}^N$ and fix a basis $(\xi_0,\cdots, \xi_{N-1})$ of $T_x\mathbb{P}^N$ such that $\xi_0 = \xi$. Now take $v\in V$ such that $p(v)=x$ and a basis $(e_0,\cdots, e_N)$ of $T_vV= V$ such that $d_vp(e_N)=0$ and $d_vp(e_i)=\xi_i$ for $i<N$. We just have to check that $(x,\left[\xi\right])$ is maped to the annouced line $\Delta$ which, with our notations, corresponds to the point $\left[e_0\wedge e_N\right]\in \mathbb{P}(\Lambda^2V^*)$.\\ 

$\begin{array}{ccccccc}
\mathbb{P}(\Omega_{\mathbb{P}^N,x}) &\to& \mathbb{P}(\Lambda^{N-1}T_x\mathbb{P}^N) &\to& \mathbb{P}(\Lambda^{N-1}V) &\to& \mathbb{P}(\Lambda^2V^*)\\
\left[ \xi_0 \right]  &\mapsto& \left[ \xi_1^*\wedge \cdots \wedge \xi_{N-1}^* \right] &\mapsto& \left[ e_1^*\wedge \cdots \wedge e_{N-1}^* \right] &\mapsto& \left[ e_0\wedge e_N \right]
\end{array}$

\end{proof}

With this we can prove our proposition,
\begin{e-proposition}\label{droite}
Let $X\subseteq \mathbb{P}^N$ be a smooth variety then $\Omega_X(2)$ is ample if and only if $X$ doesn't contain any line.
\end{e-proposition} 

\begin{proof}
By Lemma~$\ref{Han}$, we know that $\Omega_X(2)$ is ample if and only if the restriction of $\varphi$, $\varphi_{X}: \mathbb{P}(\Omega_X)\subseteq \mathbb{P}(\Omega_{\mathbb{P}^N})\to Gr(2,\mathbb{P}^N)$ is finite.

Now, if $X$ contains a line $\Delta$ then $\varphi_X$ is not finite since the curve $\mathbb{P}(K_{\Delta})~\subseteq~\mathbb{P}(\Omega_X)$ gets maped to the point in $Gr(2,\mathbb{P}^N)$ representing $\Delta$.

If $\varphi_{X}$ is not finite then there is a curve $C\subseteq \mathbb{P}(\Omega_X)$ which gets maped to a point in $Gr(2,\mathbb{P}^N)$ corresponding to a line $\Delta$ in $\mathbb{P}^N$. Let $\Gamma=\pi(C)$, Lemma~$\ref{Han}$ tells us that the embedded tangent space $\mathbb{T}_x\Gamma$ equals $\Delta$ for all $x\in \Gamma$ and therefore  $\Delta\subseteq X$.   
\end{proof}
Recall also that in~\cite{Cle86}, H. Clemens Proves that a generic hypersurface $H\subset \mathbb{P}^N$ of degree greater than $2N-1$ contains no rational curves, and therefore $\Omega_H\otimes \mathcal{O}_{\mathbb{P}^N}(2)$ is ample.  \\

\textit{Acknowledgement}. The author wishes to thank C. Mourougane for his many advices, and S. Diverio for helpful discussions and comments.

\end{document}